\documentclass[10pt,a4paper]{amsart}
\usepackage[latin1]{inputenc}
\usepackage{amsmath}
\usepackage{amsfonts}
\usepackage{amssymb}
\usepackage{amsthm}
\usepackage{xypic}
\usepackage{color}

\newtheorem{pro}{Proposition}[section]
\newtheorem{teo}[pro]{Theorem}
\newtheorem{defi}[pro]{Definition}
\newtheorem{lem}[pro]{Lemma}

\newtheorem{remark}[pro]{Remark}

\newtheorem{ex}[pro]{Example}

\newcommand{\modu}{{\mathrm{mod}}}

\newcommand{\Coker}{{\mathrm{Coker}}}

\newcommand{\findim}{{\mathrm{fin.dim}}}

\newcommand{\repdim}{{\mathrm{rep.dim}}}

\newcommand{\Phidim}{{\Phi\,\mathrm{dim}}}

\newcommand{\add}{{\mathrm{add}}}
\newcommand{\rad}{{\mathrm{rad}}}

\newcommand{\D}{\mathcal{D}}
\newcommand{\X}{\mathcal{X}}

\newcommand{\K}{\mathbb{K}}

\newcommand{\pj}{\mathcal{P}}

\author{Jos\'e Armando Vivero$^{1}$}
\thanks{$^1$ Instituto de Matem\'atica y Estad\'istica Rafael Laguardia, Universidad de la Rep\'ublica, Uruguay.}

\begin{document}
\title{Triangular Lat-Igusa-Todorov algebras}
\renewcommand{\shortauthors}{J. Vivero}
\email{jvivero@fing.edu.uy}
\maketitle

\begin{abstract} Recently the authors D. Bravo, M. Lanzilotta, O. Mendoza and J. Vivero gave a generalization of the concept of Igusa-Todorov algebra and proved that those algebras, named Lat-Igusa-Todorov (LIT for short), satisfy the finitistic dimension conjecture. In this paper we explore the scope of that generalization and give conditions for a triangular matrix algebra to be LIT in terms of the algebras and the bimodule used in its definition. As an application we obtain that the tensor product of an LIT $\mathbb{K}$-algebra with a path algebra of a quiver whose underlying graph is a tree, is LIT.\\


\end{abstract}

\section{Introduction}

This work is framed in the theory of representations of artin algebras with a particular interest in the finitistic dimension ($\findim$) conjecture, which states that for a given artin algebra $\Lambda$ there is a uniform bound for all projective dimensions of all finitely generated (f.g.) $\Lambda$-modules with finite projective dimension. This conjecture can be traced back to the work of H. Bass, particularly the article \cite{Bass}. For a survey of the conjecture we refer the reader to \cite{Z}.\\

In relation to the $\findim$ conjecture, K. Igusa and G. Todorov in \cite{IT} defined the now called Igusa-Todorov functions (IT functions) and used them to prove, among other important results, that all algebras with representation dimension ($\repdim$) smaller or equal than $3$ satisfy the $\findim$ conjecture. This result is very interesting since the concept of representation dimension, defined by M. Auslander in \cite{Aus}, has been extensively studied in connection with several problems in Representation Theory. For example Auslander proved that an algebra is of finite representation type if and only if its representation dimension is smaller or equal than 2. This lead to the question of deciding if it was always finite and finding bounds for this number. In 2003 O. Iyama proved in \cite{Iyama} that the $\repdim$ is always finite, although it can be arbitrarily large (\cite{R}, \cite{Opp}). In 2009 J. Wei defined in \cite{W} the concept of Igusa-Todorov algebra (IT algebra) and proved that these algebras satisfy the $\findim$ conjecture, providing an extensive list of IT algebras such as monomial algebras, special biserial algebras, tilted algebras and algebras with radical square zero, among others. The question of whether all artin algebras are IT was answered in 2016 by T. Conde in her Ph.D. thesis \cite{TC}. She used some results proved by R. Rouquier in \cite{R} to exhibit as a counterexample a family of algebras that are not IT, but do satisfy the $\findim$ conjecture, namely exterior algebras of vector spaces of dimension greater or equal than $3$.\\ 
In \cite{BLMV} the authors defined generalized IT functions, which gave way to a generalization of the concept of Igusa-Todorov algebra. This new class of algebras, named Lat-Igusa-Todorov (LIT for short), is strictly bigger than IT algebras, since it includes all self-injective algebras, which we know are not IT in general because the example provided by Conde is a family of self-injective algebras. It turns out that among LIT algebras it is possible to find algebras that are not IT, neither self-injective, as can be seen in Example \ref{TTT_selfinjective}. Finally, in \cite{BLMV} it was proven that LIT algebras also satisfy the $\findim$ conjecture.\\

In this article we focus on triangular matrix algebras of the form $\Lambda:=\begin{pmatrix}
T & 0 \\
M & U
\end{pmatrix}$ and give conditions in terms of the algebras $T$ and $U$ and the bimodule $M$ for $\Lambda$ to be LIT. In \cite{BLM} the authors solved this problem in the context of IT algebras, so it becomes natural to wonder what happens in the more general case of LIT algebras. The conditions we found are as follows: if $T$ and $U$ are LIT algebras and $M$ is such that $M_T$ is projective, $_{U}M$ is projective and $M\otimes_TP$ is indecomposable for all $P\in\modu\,T$ indecomposable and projective, then the triangular matrix algebra $\Lambda$ is LIT. We also give some alternatives to the previous hypotheses obtaining two versions of Theorem \ref{TriangLIT}, namely Propositions \ref{TriangLIT_other_hipotheses} and \ref{IT_LIT}. 

Another motivation for the study of this type of algebras is linked to the existence of artin algebras of infinite $\Phi$-dimension (\cite{Hanson}, \cite{Barrios}). In particular, the example provided by E. Hanson and K. Igusa in \cite{Hanson} is a $\mathbb{K}$-algebra that is a tensor product of two LIT $\mathbb{K}$-algebras and in that particular case it is again an LIT algebra so it satisfies the finitistic dimension conjecture, even though its $\Phi$-dimension is infinite. A natural question that arises here is if the tensor product of two LIT algebras is again LIT. Even though we do not provide a complete answer to this question, it arises from the main theorem (Theorem \ref{TriangLIT}) that the tensor product over a field $\mathbb{K}$ of an LIT $\mathbb{K}$-algebra with a path algebra of a quiver whose underlying graph is a tree, is  indeed LIT.\\

After this introduction, the reader will find Section 2 containing the necessary background material to understand the upcoming results. In Section 3 we prove Theorem \ref{TriangLIT}, which is the central part of this work, providing alongside some examples and consequences that are of interest. Finally, in Section 4 we give some applications of Theorem \ref{TriangLIT} and there is a sub-section dedicated to tensor products, where we prove that the tensor product of an LIT $\mathbb{K}$-algebra with a path algebra of a tree is LIT, leaving the door open to a future research to investigate if the tensor product of two LIT $\mathbb{K}$-algebras is again LIT.


\section{Preliminaries}

In this section some required definitions and results are presented in order to make use of them in what follows. Unless otherwise stated, we are going to work with left modules.\\

First we give the definition of the Igusa-Todorov function $\Phi$. We present an alternative way of defining the $\Phi$ function that is equivalent to the original definition given in \cite{IT}. For an artin algebra $\Lambda$, we denote by $\pj(\Lambda)$ the full subcategory of $\modu\,\Lambda$ consisting of the projective modules. Define $K_{\pj(\Lambda)}$ to be the free abelian group generated by the set $\{[X]\}$ of isoclasses of indecomposable, non-projective f.g. $\Lambda$-modules. For every $X\in \modu\,\Lambda$, we can define the subgroup $\left\langle X\right\rangle $ as the free abelian group generated by the set of all isoclasses of indecomposable non-projective modules which are direct summands of $X$. We denote by $L$ the endomorphism of $K_{\pj(\Lambda)}$ defined by $L([X])=[\Omega(X)]$, where $\Omega$ denotes the syzygy operator. In this setting, the $\Phi$ function can be defined as follows.

\begin{defi}\label{defi:Phi}
Let $\Lambda$ be an artin algebra we define a function $\Phi:\modu\,\Lambda\rightarrow \mathbb{Z}_{\geq 0}$ in the following way:
$$\Phi(X):=\min\{n\in\mathbb{Z}_{\geq 0} : rk L^k(\left\langle X\right\rangle)=rk L^{k+1}(\left\langle X\right\rangle), \forall k\geq n\}.$$
\end{defi}
Given a class $\X\subseteq\modu\,\Lambda$ we define the $\Phi$-dimension of $\X$ as follows: $$\Phidim(\X):=\sup\{\Phi(X)\ | \ X\in\X\}.$$
For more on the Igusa-Todorov functions ($\Phi$ and $\Psi$) and their properties we refer to \cite{IT}.\\

Next we recall from \cite{BLMV} the definition of LIT algebra. 

\begin{defi}\label{LIT} An $n$-LIT algebra, where $n$ is a non-negative integer, is an artin algebra $\Lambda$ satisfying the following two conditions:
\begin{itemize}
\item[(a)]  there is some class $\D\subseteq\modu\,\Lambda$ such that $\add\,\D=\D,$  $\Omega(\D)\subseteq\D$ and $\Phidim\,(\D)=0;$
\item[(b)] there is some $V\in\modu\,\Lambda$ satisfying that each $M\in \modu\,\Lambda$ admits an exact sequence 
$$0\longrightarrow X_1\longrightarrow X_0\longrightarrow \Omega^nM\longrightarrow 0,$$ such that $X_1=V_1\oplus D_1$, $X_0=V_0\oplus D_0$, with $V_1,V_0\in \add\,V$ and $D_1,D_0\in \D.$
\end{itemize}
In case we need to specify the class $\D$ and the $\Lambda$-module $V$ in the above definition, we say that $\Lambda$ is an $(n,V,\D)$-LIT algebra.
\end{defi}

The following remarks will be used in the proof of Theorem \ref{TriangLIT}.

\begin{remark}\label{addD}
It is possible to obtain a class $\D$ as in item (a) if what we have instead is another class $\D'$ that is closed under direct sums and syzygies and $\Phidim\,(\D')=0.$ Namely the class $\D:=\add\D'$ satisfies all the conditions in (a):
\begin{itemize}
\item It is clear from the definition that $\D=\add\,\D$.
\item Take $X\in \D$, then there exists $Y\in\D$ such that $X\oplus Y\in \D'$, then $\Omega(X)\oplus\Omega(Y)\in \D'$, hence $\Omega(X)\in \add\D'=\D,$ so we get $\Omega(\D)\subseteq \D.$
\item Finally, let $X\in \D$ and $Y\in\D$ be such that $X\oplus Y\in \D'$, from the properties of the $\Phi$ function, we have $\Phi(X)\leq\Phi(X\oplus Y)=0$, so $\Phidim\,(\D)=0$.
\end{itemize}
\end{remark}

\begin{remark}\label{k,m_LIT}
We point out that, if we have algebras $T$ and $U$ which are $k$-LIT and $m$-LIT respectively, they can be regarded as $n:=\max\{k,m\}$-LIT.
\end{remark} 

We recall the definition of a triangular matrix algebra and give some of the properties that we will need. Let $T$ and $U$ be artin algebras and $M$ be a $U$-$T$-bimodule. We consider the algebra
$$\Lambda:=\begin{pmatrix}
T & 0 \\
M & U
\end{pmatrix},$$
whose product is given by multiplication of matrices.\\

For this type of algebra we have a very useful description of its module category. Any left $\Lambda$-module can be seen as a triple $(A,B,f)$, where $A$ and $B$ are left $T$- and $U$-modules respectively and $f:M\otimes_T A\rightarrow B$ is a morphism of $U$-modules (see \cite[Proposition 2.2, Page 74]{ARS}).\\
A morphism $\alpha:(A_1,B_1,f_1)\rightarrow (A_2,B_2,f_2)$ is a pair of morphisms $(\alpha_1,\alpha_2)$ such that $\alpha_1\in Hom_T(A_1,A_2)$ and $\alpha_2\in Hom_U(B_1,B_2)$ and the following diagram commutes:
$$\xymatrix{
M\otimes_T A_1\ar[r]^{1\otimes\alpha_1}\ar[d]^{f_1} & M\otimes_T A_2\ar[d]^{f_2} \\
B_1\ar[r]^{\alpha_2} & B_2
}
$$
A sequence $\small\xymatrix{0\ar[r] & (A_1,B_1,f_1)\ar[r]^{\alpha} & (A_2,B_2,f_2)\ar[r]^{\beta} & (A_3,B_3,f_3)\ar[r] & 0 }$ is exact in $\modu\,\Lambda$ if and only if the sequences $\xymatrix{0\ar[r] & A_1\ar[r] & A_2\ar[r] & A_3\ar[r] & 0 }$ and $\xymatrix{0\ar[r] & B_1\ar[r] & B_2\ar[r] & B_3\ar[r] & 0 }$ are exact in $\modu\,T$ and $\modu\,U$ respectively. In particular, if $M_T$ is flat, there is a commutative diagram with exact rows as follows
$$\xymatrix{
0\ar[r] & M\otimes_T A_1\ar[r]^{1\otimes\alpha_1}\ar[d]^{f_1} & M\otimes_T A_2\ar[d]^{f_2}\ar[r]^{1\otimes\beta_1} & M\otimes_T A_3\ar[d]^{f_3}\ar[r] & 0 \\
0\ar[r] & B_1\ar[r]^{\alpha_2} & B_2\ar[r]^{\beta_2} & B_3\ar[r] & 0
}
$$
A triple $(A,B,f)$ is an indecomposable projective $\Lambda$-module if and only if it has either the form $(P,M\otimes_T P, 1\otimes 1)$ or $(0,Q,0)$, where $P$ and $Q$ are indecomposable projective modules over $T$ and $U$ respectively.\\
For further properties and more details we recommend \cite[Chapter III, Section 2]{ARS}.\\

The following proposition describes the syzygies of modules in a triangular matrix algebra and will be paramount in all that follows.

\begin{pro}\cite[Lemma 4.2]{BLM}\label{Syzygy_triang}
Let $\Lambda:=\begin{pmatrix}
T & 0 \\
M & U
\end{pmatrix},$ such that $M$ is projective both as a left $U$-module and as a right $T$-module. Let $(A,B,f)\in \modu\,\Lambda$, then $$\Omega^n(A,B,f)=(\Omega^n(A),M\otimes P_{n-1}^A,1\otimes i_n)\oplus (0,\Omega^n(B),0),$$ where $P_{n-1}^A$ is the projective cover of $\Omega^{n-1}(A)$ and the map $i_n:\Omega^n(A)\hookrightarrow P_{n-1}^A$ is the inclusion.
\end{pro}

We finish this section by making a very elementary but useful remark.

\begin{remark}\label{Syzygy_triang_rmk}
If $(A,B,f)$ is such that $B$ is projective, then for all $n\geq 1$ we have that $\Omega^n(A,B,f)=\Omega^n(A,0,0)$.
\end{remark}

\section{Triangular LIT algebras}
In this section we give conditions for a triangular matrix algebra $\Lambda:=\begin{pmatrix}
T & 0 \\
M & U
\end{pmatrix}$ to be LIT in terms of the algebras $T$ and $U$ and the bimodule $M$. We give the following central theorem that establishes under what conditions a triangular matrix algebra is LIT. At the end of the section we prove two versions of Theorem \ref{TriangLIT} by making some changes in the hypotheses.

\begin{teo}\label{TriangLIT}
Let $T$ and $U$ be $(n,V_T,\D_T)$ and $(n,V_U,\D_U)$ LIT algebras. Let $_{U}M_T$ be projective both as a left $U$-module and as a right $T$-module and such that $M\otimes_TP$ is indecomposable whenever $P\in \pj(T)$ is indecomposable. Then
\begin{itemize}
\item[(i)] The algebra 
$$\Lambda:=\begin{pmatrix}
T & 0 \\
M & U
\end{pmatrix}$$
is $(n+1,V_{\Lambda},\D_{\Lambda})$-LIT, where the class $\D_{\Lambda}=\add\left(\Omega(\D_T,0,0)\oplus (0,\D_U,0)\right)$ and $V_{\Lambda}=\Omega(V_T,V_U,0)\oplus \Lambda$.
\item[(ii)] $\findim (\Lambda)<\infty.$
\end{itemize}
 
\end{teo}

\begin{proof}
First we point out that, because of Remark \ref{k,m_LIT} there is no loss of generality assuming that both $T$ and $U$ are $n$-LIT for the same integer. Secondly, item (ii) follows immediately from (i) and \cite[Theorem 5.4]{BLMV}. In order to prove item (i), we have to show that the class $\D_{\Lambda}$ satisfies the conditions (a) and (b) of Definition \ref{LIT}. First we will see that the conditions in item (a) are satisfied and then we deal with those in item (b).\\

\textbf{Conditions in item (a):}
We start making the following observation: the class $\D_{\Lambda}$ can be seen as $\add\,\D_{\Lambda}'$ as in Remark \ref{addD}, where $\D_{\Lambda}'=\Omega(\D_T,0,0)\oplus (0,\D_U,0)$, so it suffices to prove that $\D_{\Lambda}'$ is closed under direct sums and syzygies and $\Phidim\,(\D')=0.$\\

It is straightforward that $\D_{\Lambda}'$ is closed under direct sums. To see that it is closed under syzygies, take $\Omega(A,0,0)\oplus (0,B,0)\in \Omega(\D_T,0,0)\oplus (0,\D_U,0)$. Then, 
\begin{align*}
\Omega\left( \Omega(A,0,0)\oplus (0,B,0) \right) & = \Omega (\Omega(A),M\otimes P_0^A,1\otimes i_1)\oplus (0,\Omega(B),0)\\ & = \Omega (\Omega(A),0,0)\oplus (0,\Omega(B),0).
\end{align*}
This last term is in $\D_{\Lambda}$ because $\D_T$ and $\D_U$ are closed under syzygies.\\

We prove now the last condition, namely that $\Phidim\,(\D_{\Lambda}')=0$.\\
Take $\mathcal{M}=\Omega(A,0,0)\oplus (0,B,0)\in \D_{\Lambda}'$ and let us calculate $\Phi(\mathcal{M})$. First we need a direct sum decomposition in indecomposable modules for $\mathcal{M}$. It is clear that such decomposition for $(0,B,0)$ coincides with the decomposition of $B$ in $\modu\,U$ as a direct sum of indecomposable modules, so if $B=\oplus_{j=1}^m B_j$, then 
$$(0,B,0)=\bigoplus_{j=1}^m (0,B_j,0).$$
On the other hand, $\Omega(A,0,0)=(\Omega A,M\otimes_T P_{0}^A,1\otimes i_1)$ and we know that if $A=\oplus_{k=1}^n A_k$ as a sum of indecomposable $T$-modules, then $$(\Omega A,M\otimes_T P_{0}^A,1\otimes i_1)=\bigoplus_{k=1}^n (\Omega A_k,M\otimes_T P_0^{A_k},1\otimes i_{1}).$$
Next we prove that each of the above summands is indecomposable. We point out that the condition $M\otimes_TP$ is indecomposable whenever $P$ is indecomposable and projective is used only to prove the last affirmation. Suppose for some $k$, that $(\Omega A_k,M\otimes_T P_0^{A_k},1\otimes i_{1})=(X,Q,j)\oplus (X',Q',j')$. Because $M\otimes_TP$ is indecomposable for all $P$ indecomposable and projective, we have that $Q$ and $Q'$ are of the form $M\otimes_TP_1$ and $M\otimes_TP_2$, where $P_1\oplus P_2 = P_0^{A_k}$.\\
In addition, the map $j:M\otimes_T X\rightarrow M\otimes_TP_1$ has to be the restriction of $1\otimes i_1,$ so $j=1\otimes j_1$ and also $j'=1\otimes j_2$, where $j_1:X\rightarrow P_1$ and $j_2:X'\longrightarrow P_2$ are inclusions. In this way we get that the map $i_1:\Omega A\rightarrow P_{0}^{A_k}$ can be decomposed as $X\oplus X'\rightarrow P_1\oplus P_2$ as a diagonal map $\begin{pmatrix} j_1&0 \\ 0&j_2 \end{pmatrix}$. Because of the uniqueness of the cokernel, we get $A_k\cong \Coker\,j_1\oplus \Coker\,j_2$ and since $A_k$ is indecomposable, one of the summands is zero. Without loss of generality we may assume $\Coker\,j_2=0$, so $X'\cong P_2$, but this contradicts the minimality of the projective cover of $A_k$, unless $X'=P_2=0$. This means that in the decomposition $(\Omega A_k,M\otimes P_0^{A_k},1\otimes i_{1})=(X,Q,j)\oplus (X',Q',j')$, one summand has to vanish, therefore $(\Omega A_k,M\otimes P_0^{A_k},1\otimes i_{1})$ is indecomposable.\\

Now we have a decomposition in indecomposable direct summands for $\mathcal{M}$:
$$\mathcal{M} = \left(\bigoplus_{k=1}^n (\Omega A_k,M\otimes P_0^{A_k},1\otimes i_{1})\right)\oplus \left( \bigoplus_{j=1}^m (0,B_j,0)\right).$$
Let us calculate the rank of the group $\left\langle \mathcal{M}\right\rangle$, which is defined as the free abelian group having as a basis the set of isomorphism classes of indecomposable, non-projective direct summands of $\mathcal{M}$. In order to do so, we first observe that if an indecomposable summand of the form $(\Omega A_k,M\otimes P_0^{A_k},1\otimes i_{1})$ is isomorphic to one of the form $(0,B_j,0)$, this means they are projective in $\modu\,\Lambda$, therefore we get $$rk\left\langle \mathcal{M}\right\rangle = rk\left\langle\bigoplus_{k=1}^n (\Omega A_k,M\otimes P_0^{A_k},1\otimes i_{1})\right\rangle + rk\left\langle B\right\rangle$$
Another observation is that if $\Omega A_k \cong \Omega A_l$, then $(\Omega A_k,M\otimes P_0^{A_k},1\otimes i_{1})\cong (\Omega A_l,M\otimes P_0^{A_l},1\otimes i_{1})$. Let us see why this is true. Since $A\in \D_T$, we have that if $\Omega A_k \cong \Omega A_l$, then $A_k\cong A_l$, because otherwise we would have that $\Phi(A_k\oplus A_l)\geq 1$, which is impossible. So, because of the uniqueness of the projective cover we get the following commutative diagram:
$$\xymatrix{0\ar[r] & \Omega A_k \ar[r]\ar[d]^{\cong} & P_0^{A_k}\ar[r]\ar[d]^{\cong} & A_k\ar[d]^{\cong}\ar[r] & 0 \\
 0\ar[r] & \Omega A_l \ar[r] & P_0^{A_l} \ar[r] & A_l\ar[r] & 0}$$
 Applying the functor $M\otimes_T-$, which is exact because $M$ is projective, we obtain the commutative diagram
 $$\xymatrix{0\ar[r] & M\otimes_T\Omega A_k \ar[r]\ar[d]^{1\otimes\cong} & M\otimes_TP_0^{A_k}\ar[r]\ar[d]^{1\otimes\cong} & M\otimes_TA_k\ar[d]^{1\otimes\cong}\ar[r] & 0 \\
 0\ar[r] & M\otimes_T\Omega A_l \ar[r] & M\otimes_TP_0^{A_l} \ar[r] & M\otimes_TA_l\ar[r] & 0}$$
From the left square we get $(\Omega A_k,M\otimes P_0^{A_k},1\otimes i_{1})\cong (\Omega A_l,M\otimes P_0^{A_l},1\otimes i_{1})$.\\
In conclusion: 
\begin{align*}
rk\left\langle \mathcal{M}\right\rangle & = rk\left\langle\bigoplus_{k=1}^n (\Omega A_k,M\otimes P_0^{A_k},1\otimes i_{1})\right\rangle + rk\left\langle B\right\rangle \\
& = rk \left\langle \Omega A_1,\dots , \Omega A_n\right\rangle + rk\left\langle B\right\rangle \\
& = rk L(\left\langle A\right\rangle) + rk\left\langle B\right\rangle \\ 
& = rk\left\langle A\right\rangle + rk\left\langle B\right\rangle.
\end{align*}
Furthermore, because of the way syzygies are computed in $\modu\,\Lambda$, we have the following equality for every $n\geq 1$.
\begin{align*}
rk\,L^n(\left\langle \mathcal{M}\right\rangle) & = rk\,\left\langle \Omega^{n+1}A_1, \dots ,\Omega^{n+1}A_n \right\rangle + rk\,L^n(\left\langle B\right\rangle) \\
& = rk L^{n+1}(\left\langle A\right\rangle) + rk L^n(\left\langle B\right\rangle)\\
 & = rk\left\langle A\right\rangle + rk\left\langle B\right\rangle.
\end{align*}
The above equalities imply that $\Phi(\mathcal{M})=0$ and this finishes the part of the proof that concerns item (a) of Definition \ref{LIT}.\\

\textbf{Conditions in item (b):}

The idea for this part of the proof is to obtain short exact sequences for each of the summands of $\Omega^n(A,B,f)$ so that after we glue them together, we obtain the desired sequence for item (b).\\

Since $T$ and $U$ are $n$-LIT algebras, there exist short exact sequences like these:
$$\xymatrix{0\ar[r] & D_1\oplus V_1\ar[r]^{\alpha} & D_0\oplus V_0\ar[r]^{\beta} & \Omega^n(A)\ar[r] & 0 },$$ where $\ D_1,D_0\in \D_T,\ V_1,V_0\in \add\,V_T.$

$$\xymatrix{0\ar[r] & D'_1\oplus V'_1\ar[r]^{\alpha'} & D'_0\oplus V'_0\ar[r]^{\beta'} & \Omega^n(B)\ar[r] & 0 },$$ where $\ D'_1,D'_0\in \D_U,\ V'_1,V'_0\in \add\,V_U.$\\

Because $M_T$ is flat (it is actually projective), it induces the following commutative diagram with exact rows:
$$\xymatrix{
0\ar[r] & M\otimes_T (D_1\oplus V_1)\ar[r]^{1\otimes\alpha}\ar[d]^{0} & M\otimes_T (D_0\oplus V_0)\ar[d]^{1\otimes i_n\circ\beta}\ar[r]^{1\otimes\beta} & M\otimes_T\Omega^n(A) \ar[d]^{1\otimes i_n}\ar[r] & 0 \\
0\ar[r] & 0\ar[r]^{0} & M\otimes P_{n-1}^A\ar[r]^{id} & M\otimes P_{n-1}^A\ar[r] & 0
}
$$
From it we can obtain the following short exact sequence in $\modu\,\Lambda$:
$$\small\xymatrix{(D_1\oplus V_1,0,0)\ar@{^(->}[r]^(.35){(\alpha,0)} & (D_0\oplus V_0,M\otimes P_{n-1}^A,1\otimes i_n\circ\beta)\ar@{->>}[r]^(.55){(\beta,id)} & (\Omega^n(A),M\otimes P_{n-1}^A,1\otimes i_n)}$$

In a more direct way, another sequence is obtained for the second summand of $\Omega^n(A,B,f)$ , this time involving only the middle term. 
$$\xymatrix{(0,D'_1\oplus V'_1,0)\ar@{^(->}[r]^{(0,\alpha')} & (0,D'_0\oplus V'_0,0)\ar@{->>}[r]^{(0,\beta')} & (0,\Omega^n(B),0)}$$

We can glue together these two sequences via direct sums and obtain the following: 
$$\delta: \ \xymatrix{X\ar@{^(->}[r] & Y \ar@{->>}[r] & \Omega^n(A,B,f)},$$ where $X=(D_1\oplus V_1,D'_1\oplus V'_1,0)$, $Y=(D_0\oplus V_0,M\otimes P_{n-1}^A,1\otimes i_n\circ\beta)\oplus (0,D'_0\oplus V'_0,0).$\\

Now take the sequence $\delta$ and apply the Horseshoe Lemma: 
$$\Omega(\delta): \ \xymatrix{\Omega(X)\ar@{^(->}[r] & \Omega(Y)\oplus Q \ar@{->>}[r] & \Omega^{n+1}(A,B,f)},$$ 
where $\Omega(X)=\Omega(D_1,D'_1,0)\oplus \Omega(V_1,V'_1,0)$, $\Omega(Y)=\Omega(D_0,D'_0,0)\oplus \Omega(V_0,V'_0,0)$ and $Q\in \pj(\Lambda)$.\\

By definition we have proven that the algebra $\Lambda$ is $(n+1,V_{\Lambda},\D_{\Lambda})$-LIT, where $\D_{\Lambda}=\add\left(\Omega(\D_T,0,0)\oplus (0,\D_U,0)\right) $ and $V_{\Lambda}=\Omega(V_T,V_U,0)\oplus \Lambda$. 
\end{proof}

We now give some examples showing why this theorem is useful and finish the section with two propositions about what can be said if we weaken some hypotheses on the bimodule $M$.

\begin{ex} \label{One_point_Ext} (One point extension)\\
Let $U$ be an LIT $\mathbb{K}$-algebra and $M$ be an indecomposable projective left $U$-module. Then the one point extension $$\Lambda:=\begin{pmatrix}
\mathbb{K} & 0 \\
M & U
\end{pmatrix}$$ is LIT. The conditions on $M$ are satisfied, since the only projective indecomposable $\mathbb{K}$-module is $\mathbb{K}$ itself and $M\otimes_{\mathbb{K}}\mathbb{K}\cong M$, which is indecomposable.
\end{ex}

\begin{ex} \label{TTT} 
Let $T$ be an LIT algebra, then 
$$\Lambda:=\begin{pmatrix}
T & 0 \\
T & T
\end{pmatrix}$$
is LIT. Here $M:=T$ with its natural bimodule structure. For any indecomposable left projective module $P$, we have $M\otimes_{T}P = T\otimes_{T}P\cong P$, so we can apply Theorem \ref{TriangLIT}. We would like to point out the fact that in general $M$ does not need to be indecomposable to satisfy the conditions of the theorem.
\end{ex}

\begin{ex} \label{TTT_selfinjective}
Let $T$ be an LIT algebra that is not an Igusa-Todorov algebra (this is the case of the exterior algebras of vector spaces of dimension $\geq 3$). Then, the algebra $$\Lambda:=\begin{pmatrix}
T & 0 \\
T & T
\end{pmatrix}$$
 is LIT. However, $\Lambda$ is not self-injective and from \cite[Theorem 4.5]{BLM} we know it is not IT. In addition, as a consequence of \cite[Proposition 3.1]{W}, we have that $\repdim(\Lambda)\geq 4.$
\end{ex}

Now we turn our attention to one of the conditions we had to impose on $M$, namely that $M\otimes_TP$ is indecomposable whenever $P\in \pj(T)$ is indecomposable. In what follows we offer two ways in which this hypothesis can be substituted by another, such that we still obtain the same result as in Theorem \ref{TriangLIT}.

First we remark the fact that the hypothesis mentioned above is used only once in the proof of the previous theorem, namely it guarantees that if $A\in\D_T$ is indecomposable, then $\Omega(A,0,0)$ is indecomposable as a $\Lambda$-module. This means that a version of Theorem \ref{TriangLIT} holds if we exchange these conditions. For completion we state it next.

\begin{pro}\label{TriangLIT_other_hipotheses}
Let $T$ and $U$ be $(n,V_T,\D_T)$ and $(n,V_U,\D_U)$ LIT algebras. Let $_{U}M_T$ be projective both as a left $U$-module and as a right $T$-module and consider the triangular matrix algebra $\Lambda:=\begin{pmatrix}
T & 0 \\
M & U
\end{pmatrix}$.
\begin{itemize}
\item[(i)] If $\Omega(A,0,0)$ is indecomposable, for all $A\in\D_T$ indecomposable, then $\Lambda$ is $(n+1,V_{\Lambda},\D_{\Lambda})$-LIT, where $\D_{\Lambda}=\add\left(\Omega(\D_T,0,0)\oplus (0,\D_U,0)\right)$ and $V_{\Lambda}=\Omega(V_T,V_U,0)\oplus \Lambda$.
\item[(ii)] $\findim (\Lambda)<\infty.$
\end{itemize}
\end{pro}

Another way of removing the condition on $M$ ($M\otimes_TP$ is indecomposable whenever $P$ is indecomposable and projective) is to ask that $T$ is an $(n, V_T)$ Igusa-Todorov algebra instead of an LIT algebra. The concept of IT algebra is less general, but it is still very vast and interesting. It includes for instance the case $T$ is a field, the case $T$ has finite global dimension and the case $T$ is syzygy finite among other cases of interest. We then have the following proposition.

\begin{pro}\label{IT_LIT}
Let $T$ be an $(n, V_T)$ IT algebra and $U$ be an $(n,V_U,\D_U)$ LIT algebra. Let $_{U}M_T$ be projective as a left $U$-module and as a right $T$-module. Then
\begin{itemize}
\item[(i)] The algebra 
$$\Lambda:=\begin{pmatrix}
T & 0 \\
M & U
\end{pmatrix}$$
is $(n+1,V_{\Lambda},\D_{\Lambda})$-LIT, where $\D_{\Lambda}=(0,\D_U,0), V_{\Lambda}=\Omega(V_T,V_U,0)\oplus \Lambda.$ 
\item[(ii)] $\findim (\Lambda)<\infty.$
\end{itemize}
 
\end{pro}

\begin{proof}
First we point out that if $T$ is a $(n, V_T)$ Igusa-Todorov algebra, then $T$ is a $(n,\{0\},V_T)$ LIT algebra. Also, as in the proof of Theorem \ref{TriangLIT}, item (ii) follows immediately from (i) and \cite[Theorem 5.4]{BLMV}.

In what follows we check the two conditions in Definition \ref{LIT}, but condition (a) only involves the class $(0,\D_U,0)$, which we have already seen in the previous proof satisfies the conditions in item (a), so all we need to do is prove condition (b). In order to do so, let us use the short exact sequence $\Omega(\delta)$ obtained in the previous demonstration, which we can, because the same hypotheses we needed to obtain it hold here, so we get a sequence  
$$\Omega(\delta): \ \xymatrix{\Omega(X)\ar@{^(->}[r] & \Omega(Y)\oplus Q \ar@{->>}[r] & \Omega^{n+1}(A,B,f)},$$ 
where $\Omega(X)=\Omega(0,D'_1,0)\oplus \Omega(V_1,V'_1,0)$, $\Omega(Y)=\Omega(0,D'_0,0)\oplus \Omega(V_0,V'_0,0)$ and $Q\in \pj(\Lambda)$.\\
Here we have that $D'_1, D'_0\in \D_U$, $V_1, V_0\in \add\,V_T$ and $V_1', V_0'\in \add\,V_U$, hence this sequence can be used to show condition (b) is satisfied.
\end{proof}

\section{Applications of Theorem \ref{TriangLIT}.}

In this section we give some ways of constructing new LIT algebras using triangular matrices. Also, as a consequence of these constructions and Theorem \ref{TriangLIT} we prove that if $T$ is an LIT $\K$-algebra and $Q$ is a quiver whose underlying graph is a tree, then $\Lambda:=T\otimes_{\K}\K Q$ is LIT.

\subsection{Row/column iterations}
We start by generalizing Example \ref{TTT}, which will give us a way of constructing LIT algebras via triangular matrices. The first thing we need is the following transformations that can be applied to any triangular matrix algebra of the form $\begin{pmatrix}T&0\\M_0&T\end{pmatrix}$ or $\begin{pmatrix}T&M_0\\0&T\end{pmatrix}$, where $T$ is an artin algebra and $M_0$ is a $T$-bimodule.

\begin{itemize}
	\item \textbf{Row iteration:} We add a row by repeating one of the rows in the matrix, then  add a column of zeroes, except for the final entry, where we put a $T$.
	\item \textbf{Column iteration:} We add a column by repeating one of the columns in the matrix, then add a row of zeroes, except for the final entry, where we put a $T$.
\end{itemize}

Observe that each iteration will increase the size of the matrix by one. For example if we take $\begin{pmatrix}T&0\\M_0&T\end{pmatrix}$ and we make a row iteration using the first row we get $\begin{pmatrix} T&0&0 \\ M_0&T&0 \\ T&0&T \end{pmatrix}$. Now taking this matrix and performing a column iteration using the second column we get $\begin{pmatrix} T&0&0&0 \\ M_0&T&0&T \\ T&0&T&0 \\ 0&0&0&T \end{pmatrix}$.

\begin{defi}\label{defi:iterations}
Let $T$ be an artin algebra, $M_0$ be a $T$-bimodule and $n\geq 0$. We say that $\Gamma$ is an $n$-iteration of $(T,M_0)$ if $\Gamma$ has been obtained after $n$ row/column iterations, starting with either one of the algebras $\begin{pmatrix}T&0\\M_0&T\end{pmatrix}$ or $\begin{pmatrix}T&M_0\\0&T\end{pmatrix}$.
\end{defi}

It is straightforward to see that if $\Gamma$ is an $n$-iteration of $(T,M_0)$, then it is an $(n+2)\times (n+2)$ triangular matrix algebra of the form $\begin{pmatrix}
	U_n & 0 \\ M_n & T
\end{pmatrix}$ or $\begin{pmatrix}
	U_n & N_n \\ 0 & T
\end{pmatrix}$, where $U_n$ is an $(n-1)$-iteration and $M_n,N_n$ are $T$-$U_n$- and $U_n$-$T$-bimodules respectively. We denote by $e_j$, with $j=1,\dots,n+2$ the element in $\Gamma$ that has $1_T$ in the entry $(j,j)$ and zeroes elsewhere. Using the fact that $M_n(N_n)$ is a row(column) repetition of $U_n$, there exist integers $k,l$ such that $M_n=e_kU_n$ and $N_n=U_ne_l.$\\ 

Let us give the following definition.

\begin{defi}
	We say that an $n$-iteration is \textbf{perfect} if $M_0$ is projective both as a left and right $T$-module and for every $i=0,\dots,n$ we have that $M_i$(resp. $N_i$) are such that if $P$ is an indecomposable projective left $U_i$-($T$-)module, then $M_i\otimes_{U_i}P$($N_i\otimes_T P$) is indecomposable as a left $T$-($U_i$-)module.
\end{defi}

It is not difficult to see that if $M_0$ is projective both as a left and right $T$-module and $\Gamma$ is an $n$-iteration, then $M_i(N_i)$ is projective both as a left $T$-($U_i$-)module and as a right $U_i$-($T$-)module for all $i=0,\dots,n.$ To see this, write $M_i=e_{k_i}U_i$, so it is projective as a right $U_i$-module. As a left $T$-module it is a direct sum of summands that are either $0, T$ or $M_0$ so it is also projective. The same argument can be adapted to $N_i.$\\

The above paragraph can be used to give an equivalent definition of perfect iteration.

\begin{pro}
	 An $n$-iteration is perfect if and only if the $0$-iteration $\begin{pmatrix}
		T & 0 \\ M_0 & T
	\end{pmatrix}$ is such that $M_0$ satisfies the hypotheses of Theorem \ref{TriangLIT} and for every $i=1,\dots,n$, $M_i(N_i)$ also satisfy those hypotheses.
\end{pro}

We have the following lemma concerning iterations.

\begin{lem}\label{iterations_indecomp}
	Let $T$ be an artin algebra, $M_0$ be a $T$-bimodule and $\Gamma$ be a perfect $n$-iteration of $(T,M_0)$. If $L$ is an indecomposable projective $\Gamma$-module, then $e_j\cdot L$ is an indecomposable left $T$-module for all $j=1,\dots,n+2.$ 
\end{lem}

\begin{proof}
We will proceed by induction on $n$. If $n=0$ we have that $\Gamma$ has to be either $\begin{pmatrix} T & 0 \\ M_0 & T\end{pmatrix}$ or $\begin{pmatrix} T & M_0 \\ 0 & T\end{pmatrix}$. Without loss of generality we may assume $\Gamma=\begin{pmatrix} T & 0 \\ M_0 & T\end{pmatrix}$. If $L$ is an indecomposable projective $\Gamma$-module it has the form $(P,M_0\otimes_T P,1\otimes 1)$ or $(0,Q,0)$, where $P,Q$ are indecomposable projective $T$-modules. On the other hand, we know that the triple that is identified with $L$ has the form $(e_1L, e_2L, f)$. From the fact that the iteration is perfect we obtain that both $e_1L$ and $e_2L$ are indecomposable as $T$-modules because $P,Q$ and $M_0\otimes_T P$ are indecomposable.

Assume $\Gamma$ is a perfect $n$-iteration and that the result is valid for all perfect $(n-1)$-iterations. We know that $\Gamma$ is either $\begin{pmatrix} U_n & 0 \\ M_n & T \end{pmatrix}$ or $\begin{pmatrix} U_n & N_n \\ 0 & T \end{pmatrix}$, where $U_n$ is a perfect $(n-1)$-iteration. Without loss of generality we can assume $\Gamma=\begin{pmatrix} U_n & N_n \\ 0 & T \end{pmatrix}$, the lower triangular case is analogous. Assume $L$ is an indecomposable projective $\Gamma$-module, then it either has the form $(P,N_n\otimes_T P, 1\otimes 1)$ or $(0,Q,0)$, where $P$ is an indecomposable projective $T$-module and $Q$ is an indecomposable projective $U_n$-module. On the other hand, the triple that is identified with $L$ has the form $(e_{n+2} L, (e_1+\cdots+e_{n+1})L, h)$. The first component, $e_{n+2} L$ is isomorphic to either $0$ or $P$ so it is indecomposable as a $T$-module. The second component, $(e_1+\cdots+e_{n+1})L$ is isomorphic to either $N_n\otimes_T P$ or $Q$ and since $\Gamma$ is perfect both are indecomposable projective $U_n$-modules. Applying the induction hypothesis we get $e_j(e_1+\cdots+e_{n+1})L=e_jL$ is indecomposable for $j=1,\dots,n+1$.
\end{proof}

\begin{lem}\label{iterations_perfect}
	Let $T$ be an artin algebra and $M_0$ a $T$-bimodule such that it is projective both as a left and right $T$-module and $M_0\otimes_TP$ is indecomposable whenever $P\in\pj(T)$ is indecomposable. Then all $n$-iterations are perfect for any $n\geq 0$.
\end{lem}

\begin{proof}
	We will proceed by induction on $n$. If $n=0$ it is clear form the definition of perfect iteration.\\
Assume now that all $(n-1)$-iterations of $(T,M_0)$ are perfect and let $\Gamma$ be an $n$-iteration of the form $\begin{pmatrix}
	U_n & 0 \\ M_n & T
\end{pmatrix}$ or $\begin{pmatrix}
	U_n & N_n \\ 0 & T
\end{pmatrix}$, where $U_n$ is a $(n-1)$-iteration. Let us divide the proof in these two cases.\\
\underline{Case 1:} $\Gamma=\begin{pmatrix}
	U_n & 0 \\ M_n & T
\end{pmatrix}$. What we need to prove is that for each indecomposable projective $U_n$-module $P$, $M_n\otimes_{U_n} P$ is an indecomposable $T$-module. Because of the definition of iteration, we know that $M_n$ is the repetition of one of the rows of $U_n$, this means that $M_n=e_jU_n$, for some $j=1,\dots,n+1$. It is not difficult to see that the map $e_ju\otimes p\mapsto e_jup$ defines an isomorphism $e_jU_n\otimes_{U_n} P\simeq e_jP$ of $T$-modules. Since $U_n$ is perfect and $P$ is an indecomposable projective we can use Lemma \ref{iterations_perfect} to conclude that $e_jP$ is indecomposable for all $i=1,\dots,n+1.$\\
\underline{Case 2:} $\Gamma=\begin{pmatrix}
	U_n & N_n \\ 0 & T
\end{pmatrix}$. What we need to prove is that  for each $Q$ indecomposable projective $T$-module, $N_n\otimes_T Q$ is indecomposable as a $U_n$-module. We know that $U_n$ is a perfect $(n-1)$-iteration and without loss of generality we may assume $U_n=\begin{pmatrix}U_{n-1}&0\\M_{n-1}&T\end{pmatrix}$, where $M_{n-1}=e_kU_{n-1}$ for some $k=1,\dots,n$. We know from the definition of iteration that $N_n$ is a column repetition of $U_n$ so it has to be either $\begin{pmatrix}
0 \\ T
\end{pmatrix}$ or $\begin{pmatrix}
U_{n-1}e_r\\ e_kU_{n-1}e_r
\end{pmatrix}$ for some $r=1,\dots,n.$ If $N_n=\begin{pmatrix}
0 \\ T
\end{pmatrix}$, we get $N_n\otimes_T Q=(0,Q,0)$ in $\modu\,U_n$, which is indecomposable. If $N_n=\begin{pmatrix}
U_{n-1}e_r\\ e_kU_{n-1}e_r
\end{pmatrix}$, then $N_n\otimes_T Q=(A,B,f)$ with $A=U_{n-1}e_r\otimes_T Q$, $B=e_kU_{n-1}e_r\otimes_T Q$ and $$f:e_kU_{n-1}\bigotimes_{U_{n-1}}\left( U_{n-1}e_r\otimes_T Q\right)\rightarrow  e_kU_{n-1}e_r\otimes_T Q$$ is given by $e_ku\otimes (u'e_r\otimes q)\mapsto e_kuu'e_r\otimes q.$ Now, $B$ is indecomposable because it is either $0$, $Q$ or $M_0\otimes_TQ$. From the induction hypothesis all $(n-1)$-iterations are perfect so in particular $\begin{pmatrix}
U_{n-1} & U_{n-1}e_r \\ 0 & T
\end{pmatrix}$ is perfect, hence $A=U_{n-1}e_r\otimes_T Q$ is indecomposable. Moreover, if $A$ and $B$ are nonzero, then it can be seen that $f$ is also nonzero. All these facts together allow us to conclude that $(A,B,f)$ is indecomposable as we wanted to prove.
\end{proof}

The previous lemma gives us another characterization of perfect iteration.

\begin{pro}\label{pro_perfect_iterations}
	An $n$-iteration of $(T,M_0)$ is perfect if and only if the $0$-iteration is perfect. 
\end{pro}

\begin{teo}\label{iterations_LIT}
	Let $T$ be an LIT artin algebra and $M_0$ a $T$-bimodule such that the $0$-iteration of $(T,M_0)$ is perfect, then all $n$-iterations are LIT for any $n\geq 0$.
\end{teo}

\begin{proof}
	First we observe that if the $0$-iteration is perfect, then by Theorem \ref{TriangLIT} $\begin{pmatrix}
		T & 0 \\ M_0 & T
	\end{pmatrix}$ and $\begin{pmatrix}
	T & M_0 \\ 0 & T
\end{pmatrix}$ are LIT.
	Now assume all $(n-1)$-iterations are LIT and take any $n$-iteration $\Gamma=\begin{pmatrix}
		U_n & 0 \\ M_n & T
	\end{pmatrix} $ or $ \Gamma=\begin{pmatrix}
	U_n & N_n \\ 0 & T
\end{pmatrix}$. Since the $0$-iteration is perfect, using Proposition \ref{pro_perfect_iterations} we get that $\Gamma$ is perfect. In addition, $U_n$ is LIT, $T$ is LIT and $M_n, N_n$ satisfy the hypotheses of Theorem \ref{TriangLIT}, so we can conclude that $\Gamma$ is LIT.
\end{proof}

\subsection{Tensor products}
The interest in investigating whether the tensor product of LIT algebras ia again LIT is mainly motivated by an example given in \cite{Hanson}, where the authors considered the following quivers 
$$Q:\xymatrix{& 2\ar[dl]\ar[dr] \\ 3\ar[dr] & & 4\ar[dl] \\ & 1\ar[uu]} \ \ \ \ \ \ \ \ \ \ \ \ C_3:\xymatrix{& 1\ar[dl] \\ 3\ar[rr] & & 2\ar[ul]}$$
and used them to define algebras $T:=\frac{\K Q}{\rad^2\K Q}$ and $U:=\frac{\K C_3}{\rad^2\K C_3}$. The main result of the cited paper is that the algebra $\Lambda:=T\otimes_{\K}U$ has infinite $\Phi$-dimension. However, both $T$ and $U$ are LIT and in this case one can see that $\Lambda$ is also LIT because $\rad^3\Lambda=0$ \cite[Corollary 3.5]{W}, so despite having infinite $\Phi$-dimension it does satisfy the finitistic dimension conjecture.\\

In what follows we will show that, if $T$ is an LIT algebra and $Q$ is a quiver whose underlying graph is a tree, then $\Lambda:=T\otimes_{\K}\K Q$ is LIT. For such a quiver we can always rearrange the vertices $\{1,\dots,m\}$ so that we get a quiver that looks like this:
$$\xymatrix{& & 1\ar@{-}[lld]\ar@{-}[ld]\ar@{-}[rd] & & \\
	2\ar@{-}[d] & 3\ar@{-}[d] & \cdots & k_1\ar@{-}[d] \\
	\vdots\ar@{-}[d] & \vdots\ar@{-}[d] & \vdots\ar@{-}[d] & \vdots \\
	\cdots & m-1 & m}$$
For any $\K$-algebra $T$ we have that $T\otimes_{\K} \K Q$ is a matrix subalgebra of $M_m(T)$ whose entries are either $T$ or $0$. We can construct that matrix algebra as follows: we take vertex 2 and look at the only edge that connects it to 1. If the arrow goes $1\rightarrow 2$ we write $\begin{pmatrix} T & 0 \\ T & T \end{pmatrix}$ and if the arrow goes $2\rightarrow 1$ we write $\begin{pmatrix} T & T \\ 0 & T \end{pmatrix}$. Let us call that algebra $\Gamma_0$. Now take vertex 3, there is only one edge that connects 3 to $j\in\{1,2\}$. If the arrow goes $j\rightarrow 3$ we write $\begin{pmatrix} \Gamma_0 & 0 \\ e_j\Gamma_0 & T \end{pmatrix}$ and if the arrow goes $3\rightarrow j$ we write $\begin{pmatrix} \Gamma_0 & \Gamma_0e_j \\ 0 & T \end{pmatrix}$. If one continues with this process until arriving to vertex $m$ the matrix algebra corresponding to $T\otimes_{\K} \K Q$ will be obtained. If now we check Definition \ref{defi:iterations} we see that for $m\geq 1$, the algebra $T\otimes_{\K} \K Q$ is actually an $(m-1)$-iteration of $(T,T).$ Using Theorem \ref{iterations_LIT} we get the following 

\begin{teo}\label{teo:tensor_tree}
	Let $T$ be an LIT $\K$-algebra and $Q$ a quiver whose underlying graph is a tree, then $\Lambda:=T\otimes_{\K}\K Q$ is LIT.
\end{teo}

\begin{remark}
	At this point the reader may be wondering what happens to the opposite algebra of an LIT algebra. This question, as simple as it is to formulate does not have an answer as far as I know, but the notion of LIT algebra is very recent so more research is definitely needed. The issue that makes this difficult is that the best way we know of relating the module category of an algebra $\Lambda$ with that of the opposite algebra $\Lambda^{op}$ is through the duality functor $D:\modu\,\Lambda\rightarrow\modu\,\Lambda^{op}$. Unfortunately this functor does not preserve projectives and does not commute with syzygies, so in general if $\Lambda$ is LIT, then $D(\D_{\Lambda})$ is not a suitable class for $\Lambda^{op}$ to be LIT, as one would expect. 
\end{remark}

\begin{ex}
	Let $Q_k$ be a quiver of the form $\xymatrix{
1 \ar@<1ex>[r]^{\alpha_1}_{\cdots} \ar@<-1ex>[r]_{\alpha_k} & 2 }$, where $k\geq 2$. Let $T$ be an LIT $\K$-algebra and consider $\Lambda_k=T\otimes_{\K}\K Q_k$. We will show that $\Lambda_k$ is LIT.\\

Since $Q_k$ is not a tree, we cannot apply Theorem \ref{teo:tensor_tree}. However, $\Lambda_k$ is a triangular matrix algebra of the form $\begin{pmatrix}
		T & 0 \\ T^k & T
	\end{pmatrix}$ so maybe instead we could use Theorem \ref{TriangLIT}. On one hand it is clear that $T^k$ is projective both as a left and right $T$-module, but if $P\in\modu\,T$ is indecomposable and projective, we have that $T^k\otimes_TP=P^k$, which is not indecomposable as a left $T$-module because $k\geq 2$. As the reader may have remembered, there are two versions of Theorem \ref{TriangLIT} that do not use the above condition on $M$. Since $T$ is an LIT algebra, it might as well not be IT, so we cannot use Proposition \ref{IT_LIT}. Our only hope is to be able to prove that if $A\in\D_T$ is indecomposable, then $\Omega(A,0,0)$ is indecomposable as a $\Lambda_k$-module and then make use of Proposition \ref{TriangLIT_other_hipotheses}. Fortunately this is the case, so let us prove the following\\
	
\textbf{Affirmation:} If $A\in\modu\,T$ is indecomposable, then $\Omega(A,0,0)$ is indecomposable as a $\Lambda_k$-module. In particular it is true for $A\in\D_T$ indecomposable.\\
	
First we write $\Omega(A,0,0)=(\Omega\,A, T^k\otimes_TP_0^A,1\otimes i)$ by Lemma \ref{Syzygy_triang}. Now assume that $$(\Omega\,A, T^k\otimes_TP_0^A,1\otimes i) = (X,Q,j)\oplus (X',Q',j'),$$ where $j:T^k\otimes_TX\rightarrow Q$ and $j':T^k\otimes_TX'\rightarrow Q'.$ We observe that both $j$ and $j'$ are inclusions because they have to be restrictions of $1\otimes i$, also we have $T^k\otimes_TX\cong X^k$, $T^k\otimes_TX'\cong X'^k$. Next we decompose $j|_{X}:X\rightarrow Q_1\oplus Q_2$, such that $p\circ j|_{X}:X\rightarrow Q_1$ is a left minimal version of $j|_{X}$, where $p:Q\rightarrow Q_1$ is the projection. We do the same with $j'|_{X'}$ and write $j'|_{X'}:X'\rightarrow Q_1'\oplus Q_2'$, where $p'\circ j'|_{X'}:X'\rightarrow Q_1'$ is a left minimal version of $j'|_{X'}$. In this way, we get that the map $i:\Omega(A)\rightarrow P_0^A$ can be decomposed in $X\oplus X'\rightarrow Q_1\oplus Q_1'$ as a diagonal map $\begin{pmatrix} j|_{X}&0 \\ 0&j'|_{X'} \end{pmatrix}$. This implies that actually $Q=Q_1^k$ and $Q'=Q_1'^k$ since $Q\oplus Q'=(P_0^A)^k$. Furthermore, because of the uniqueness of cokernels we must have $A\cong \Coker\,j|_{X} \oplus \Coker\,j'|_{X'}$, but $A$ is indecomposable so without loss of generality we may assume $\Coker\,j|_{X}=0$, hence $X\cong Q_1$ and this contradicts the minimality of the projective cover, unless $X=Q_1=0$. This proves that $\Omega(A,0,0)$ is indecomposable and we can apply Proposition \ref{TriangLIT_other_hipotheses} to conclude that $\Lambda_k$ is LIT for all $k\geq 2$.
\end{ex}

The previous example shows that hopefully it will be possible to prove that, if $Q$ is a finite quiver without oriented cycles and $T$ is an LIT $\K$-algebra, then $T\otimes_{\K}\K Q$ is LIT. Moreover, it leaves the door open to investigate in general if the tensor product of two LIT algebras is LIT.

\begin{center}

\end{center}

\section*{Declarations} 
\begin{itemize}
\item The author did not receive financial support from any organization for the submitted work.
\item The author has no relevant financial or non-financial interests to disclose.
\item The author has no conflicts of interest to declare.
\item Data sharing not applicable to this article as no datasets were generated or analysed during the current study.
\end{itemize}

\bibliographystyle{unsrt}

\begin{center}

\end{center}

\end{document}